\documentclass[12pt]{article}
\usepackage{amsmath,amssymb}
\usepackage{amsthm}

\def\R{\mathbf{R}}

\def\1{\mathbf{1}}

\def\al{\alpha}
\def\be{\beta}
\def\pa{\partial}

\def\de{\delta}

\def\ka{\varkappa}

\newtheorem{prop}{Proposition}[section]
\newtheorem{theorem}{Theorem}[section]

\newtheorem{remark}{Remark}

\newcommand{\la}{\lambda}

\newcommand{\om}{\omega}

\begin{document}
\title{Mean-field-game model for Botnet defense in Cyber-security}
\author{
V. N. Kolokoltsov\thanks{Department of Statistics, University of Warwick,
 Coventry CV4 7AL UK,  Email: v.kolokoltsov@warwick.ac.uk and associate member  of Institute of Informatics Problems, FRC CSC RAS}
 and A. Bensoussan\thanks{School of Management, The University of Texas at Dallas, P.O. Box 830688, SM 30 Richardson US, Graduate School 
 of Business, The Hong Kong Polytechnic University, Hung Hom, Kowloon, Hong Kong, and Ajou University, Korea  }}
\maketitle

\begin{abstract}
We initiate the analysis of the response of computer owners to various offers of defence systems
 against a cyber-hacker (for instance, a botnet attack),
as a stochastic game of a large number of interacting agents.
We introduce a simple mean-field game that models their behavior.
It takes into account both the random process of the propagation
of the infection (controlled by the botner herder) and
the decision making process of customers.
Its stationary version turns out to be exactly solvable (but not at all trivial)
under an additional natural assumption that the execution time
of the decisions of the customers (say, switch on or out the defence system)
is much faster that the infection rates.
\end{abstract}

{\bf Mathematics Subject Classification (2010)}:

\section{Introduction}

A botnet, or zombie network, is a network of computers infected with a malicious program that
 allows cybercriminals to control the infected machine remotely without the user's knowledge.
Botnets have become a source of income for entire groups of cybercriminals since the cost of
 running botnets is cheap and the risk of getting caught is relatively small due to the fact
  that other people's assets are used to launch attacks. The interactive process of the attackers
and defenders can be modeled as a Game. The use of game theory in modeling attacker-defender has
been extensively adopted in the computer security domain recently; see \cite{BenKaHo}, 
\cite {LiLiSt} and \cite{LyWi} and bibliography
there for more details.  Two aspects are important. The first one is the contamination effect.
The second one is the large number of computers. So, in fact, one deals with a stochastic game
 of a large number of interacting agents. This is amenable to Mean Field theory. To investigate
this approach represents the main objective of this paper. Our model takes into account both
 the random process of the propagation of the infection (controlled by the botnet herder) and
  the decision making process of customers. We develop a stationary version which turns out to
be exactly solvable (but not at all trivial) under an additional natural assumption that the
 execution time of the decisions of the customers (say, switch on or out the defense system)
  is much faster that the infection rates.
  
Similar models can be applied to the analysis of defense against a biological weapon, for instance by adding
the active agent (principal interested in spreading the disease), into the general mean-field epidemic model
of \cite{LiuTaYa} that extends the well established SIS (susceptible-infectious-susceptible) and
SIR (susceptible-infectious-recovered) models.

Mean-field games present a quickly developing area of the game theory.
It was initiated by Lasry-Lions \cite{LL2006} and Huang-Malhame-Caines
\cite{HCM3}, \cite{HCM07}, 
see \cite{Ba13}, \cite{BenFr}, \cite{GLL2010}, \cite{Gomsurv}, \cite{Cain14} for recent surveys,
as well as \cite{Car15}, \cite{Car13}, \cite{CarD13}, \cite{BasRa14}, \cite{KTY14}, 
\cite{TemBas14}, \cite{BaTemBa} and references therein.
The papers \cite{Gom10} and \cite{Gom13} initiated the study of finite-state space mean-field games
that are the objects of our analysis here.

The paper is organized as follows.
In the next section we introduce our model, formulate the basic mean-field game (MFG)
consistency problem in its dynamic and stationary versions leading to precise formulation of our main
problem of characterizing the stable solutions (equilibria) of the stationary problem.
This problem is a consistency problem between an HJB equation for a stochastic control of
individual players and a fixed point problem for an evolutionary dynamics. These two
preliminary problems are fully analyzed in Sections \ref{secHJB} and \ref{secfixedpoint} respectively.
Section \ref{secMFG} is devoted to the final synthesis of the stationary MFG problem from
the solutions to these two preliminary problems. In particular, the phase transitions and the bifurcation points
changing the number of solutions are explicitly found.
In the last section further perspectives are discussed.

\section{The model}

Assume that any computer can be in $4$ states: $DI, DS, UI, US$, where the first letter,
 $D$ or $U$, refers to the state of a defended
(by some system, which effectiveness we are trying to analyze) or an unprotected computer,
and the second letter, $S$ and $I$, to susceptible or infected state. The change between $D$ and $U$ is subject to
the decisions of computer owners (though the precise time of the execution of her intent is noisy)
 and the changes between $S$ and $I$ are random with distributions depending on the level of efforts
$v_H$ of the Herder and the state $D$ or $U$ of the computer.

Let $n_{DI}, n_{DS}, n_{UI}, n_{US}$ denote the numbers of computers in the corresponding states
with $N=n_{DS}+n_{DI}+n_{UI}+n_{US}$ the total number of computers.
By a state of the system we shall mean either the $4$-vector $n=(n_{DI}, n_{DS}, n_{UI}, n_{US})$
or its normalized version $x=(x_{DI}, x_{DS}, x_{UI}, x_{US})=n/N$.
 The fraction of defended computers $x_{DI}+x_{DS}$ represents the analogue of the
 control parameter $v_D$ from \cite{BenKaHo}, the level of defense of the system,
 though here it results as a compound effect of individual decisions of all players.

 The control parameter $u$ of each player may have two values, $0$ and $1$, meaning that the player is happy
 with the level of defense ($D$ or $I$) or she prefers to switch one to another. When the updating decision $1$
 is made, the updating effectively occurs after some exponential time with the parameter $\la$
 (measuring the speed of the response of the defense system). The limit $\la \to \infty$ corresponds to
 the immediate execution.

 The recovery rates (the rates of change from $I$ to $S$) are given constants $q^D_{rec}$ and $q^U_{rec}$
 for defended and unprotected computers respectively, and the rates of infection from the direct attacks
 are $v_Hq^D_{inf}$ and $v_Hq^U_{inf}$ respectively with constants $q^D_{inf}$ and $q^U_{inf}$.
 The rates of infection spreading from infected to susceptible computers are $\be_{UU}/N, \be_{UD}/N, \be_{DU}/N, \be_{DD}/N$,
 with numbers $\be_{UU}, \be_{UD}, \be_{DU}, \be_{DD}$,
 where the first (resp second) letter in the index refers to the state of the infected (resp. susceptible) computer
 (the scaling $1/N$ is necessary to make the rates of unilateral changes and binary interactions comparable in the
 $N\to \infty$ limit).

 Thus if all computers use the strategy $u_{DS}, u_{DI}, u_{US}, u_{UI}$, $u\in \{0,1\}$ and the level of attack is $v_H$,
the evolution of the frequencies $x$ in the limit $N\to \infty$ can be described by the following system of ODE:

 \begin{equation}
 \label{eqmainkineticbotnet}
 \left\{
 \begin{aligned}
& \dot x_{DI} =x_{DS} q_{inf}^D v_H -x_{DI} q_{rec}^D
 +x_{DI}x_{DS} \be_{DD}+x_{UI}x_{DS}\be_{UD}+\la (x_{UI}u_{UI}-x_{DI}u_{DI}), \\
& \dot x_{DS} =-x_{DS} q_{inf}^D v_H +x_{DI} q_{rec}^D
 -x_{DI}x_{DS} \be_{DD}-x_{UI}x_{DS}\be_{UD}+\la (x_{US}u_{US}-x_{DS}u_{DS}), \\
& \dot x_{UI} =x_{US} q_{inf}^U v_H -x_{UI} q_{rec}^U
 +x_{DI}x_{US} \be_{DU}+x_{UI}x_{US}\be_{UU}-\la (x_{UI}u_{UI}-x_{DI}u_{DI}), \\
& \dot x_{US} =-x_{US} q_{inf}^U v_H +x_{UI} q_{rec}^U
 -x_{DI}x_{US} \be_{DU}-x_{UI}x_{US}\be_{UU}-\la (x_{US}u_{US}-x_{DS}u_{DS}).
\end{aligned}
\right.
 \end{equation}

\begin{remark}
If all $\be_{UD}, \be_{UU}, \be_{DU}, \be_{UU}$ are equal to some $\be$, $q^D_{inf}=q^U_{inf}=q^{inf}$ and $q^D_{rec}=q^D_{rec}=v_D$,
where $v_D$ is interpreted as the defender group's combined defense effort, then summing up the first and the third equations
in \eqref{eqmainkineticbotnet} leads to the equation
\begin{equation}
 \label{eqmainkineticbotnetBen}
 \dot x =q_{inf} v_H (1-x)+\be x (1-x) -v_D x,
 \end{equation}
 for the total fraction of infected computers $x=x_{DI}+x_{UI}$. This equation coincides (up to some constants)
  with equation (2) from \cite{BenKaHo},
 which is the starting point of the analysis of paper \cite{BenKaHo}.
\end{remark}

It is instructive to see, how evolution \eqref{eqmainkineticbotnet} can be deduced rigorously as the limit of the Markov processes
specifying the random dynamics of $N$ players.
The generator of this Markov evolution on the states $n$ is (where  the unchanged values in the arguments of $F$ on the r.h.s
 are omitted)
 \[
 L_NF(n_{DI}, n_{DS}, n_{UI}, n_{US})=n_{DS} q_{inf}^D v_H F(n_{DS}-1, n_{DI}+1)+n_{US} q_{inf}^U v_H F(n_{US}-1, n_{UI}+1)
 \]
 \[
 + n_{DI} q_{rec}^D F(n_{DI}-1, n_{DS}+1)+n_{UI} q_{rec}^U F(n_{UI}-1, n_{US}+1)
 \]
 \[
 +n_{DI}n_{DS} \be_{DD} F(n_{DS}-1, n_{DI}+1)/N+n_{DI}n_{US} \be_{DU} F(n_{US}-1, n_{UI}+1)/N
\]
\[
 +n_{UI}n_{DS} \be_{UD} F(n_{DS}-1, n_{DI}+1)/N+n_{UI}n_{US} \be_{UU} F(n_{US}-1, n_{UI}+1)/N
\]
\[
+\la n_{DS} u_{DS} F(n_{DS}-1, n_{US}+1) + \la n_{US} u_{US} F(n_{US}-1, n_{DS}+1)
\]
\[
+\la n_{DI} u_{DI} F(n_{DI}-1, n_{UI}+1) + \la n_{UI} u_{UI} F(n_{UI}-1, n_{DI}+1),
\]
or in terms of $x$ as
\[
 L_NF(x_{DI}, x_{DS}, x_{UI}, x_{US})=Nx_{DS} q_{inf}^D v_H F(x-e_{DS}/N+e_{DI}/N)+Nx_{US} q_{inf}^U v_H F((x-e_{US}/N+e_{UI}/N)
 \]
 \[
 + Nx_{DI} q_{rec}^D F(x-e_{DI}/N+e_{DS}/N)+Nx_{UI} q_{rec}^U F(x-e_{UI}/N+e_{US}/N)
 \]
 \[
 +Nx_{DI}x_{DS} \be_{DD} F(x-e_{DS}/N+e_{DI}/N)+Nx_{DI}x_{US} \be_{DU} F(x-e_{US}/N+e_{UI}/N)
\]
\[
 +Nx_{UI}x_{DS} \be_{UD} F(x-e_{DS}/N+e_{DI}/N)+Nx_{UI}x_{US} \be_{UU} F(x-e_{US}/N+e_{UI}/N)
\]
\[
+N\la x_{DS} u_{DS} F(x-e_{DS}/N+e_{US}/N) + N\la x_{US} u_{US} F(x-e_{US}/N+e_{DS}/N)
\]
\begin{equation}
 \label{eqmainlimgenbotnet0}
+N\la x_{DI} u_{DI} F(x-e_{DI}/N+e_{UI}/N) + N\la x_{UI} u_{UI} F(x-e_{UI}/N+e_{DI}/N),
\end{equation}
where $\{e_j\}$ is the standard basis in $\R^4$.

If $F$ is a differentiable function, the generator $L_N$ turns to the generator
\[
 LF(x_{DI}, x_{DS}, x_{UI}, x_{US})=x_{DS} q_{inf}^D v_H \left(\frac{\pa F}{\pa x_{DI}}-\frac{\pa F}{\pa x_{DS}}\right)
 +x_{US} q_{inf}^U v_H \left(\frac{\pa F}{\pa x_{UI}}-\frac{\pa F}{\pa x_{US}}\right)
 \]
 \[
 + x_{DI} q_{rec}^D \left(\frac{\pa F}{\pa x_{DS}}-\frac{\pa F}{\pa x_{DI}}\right)
 +x_{UI} q_{rec}^U \left(\frac{\pa F}{\pa x_{US}}-\frac{\pa F}{\pa x_{UI}}\right)
 \]
 \[
 +x_{DI}x_{DS} \be_{DD} \left(\frac{\pa F}{\pa x_{DI}}-\frac{\pa F}{\pa x_{DS}}\right)
 +x_{DI}x_{US} \be_{DU} \left(\frac{\pa F}{\pa x_{UI}}-\frac{\pa F}{\pa x_{US}}\right)
\]
\[
 +x_{UI}x_{DS} \be_{UD} \left(\frac{\pa F}{\pa x_{DI}}-\frac{\pa F}{\pa x_{DS}}\right)
 +x_{UI}x_{US} \be_{UU} \left(\frac{\pa F}{\pa x_{UI}}-\frac{\pa F}{\pa x_{US}}\right)
\]
\[
+\la x_{DS} u_{DS} \left(\frac{\pa F}{\pa x_{US}}-\frac{\pa F}{\pa x_{DS}}\right)
 + \la x_{US} u_{US} \left(\frac{\pa F}{\pa x_{DS}}-\frac{\pa F}{\pa x_{US}}\right)
\]
\begin{equation}
 \label{eqmainlimgenbotnet}
+\la x_{DI} u_{DI} \left(\frac{\pa F}{\pa x_{UI}}-\frac{\pa F}{\pa x_{DI}}\right)
 + \la x_{UI} u_{UI} \left(\frac{\pa F}{\pa x_{DI}}-\frac{\pa F}{\pa x_{UI}}\right)
\end{equation}
in the limit $N\to \infty$.
This is a first order partial differential operator.
Its characteristics are given precisely by the ODE
\eqref{eqmainkineticbotnet}. A rigorous derivation showing the solutions
to \eqref{eqmainkineticbotnet} describe the limit of the Markov chain
generated by \eqref{eqmainlimgenbotnet0} can be found e.g. in \cite{Ko12}.   

We shall now use the
 Markov model above to assess the actions of individual players.

If $x(t)$ and $v_H(t)$ are given, the dynamics of each individual player is the Markov chain on $4$ states
with the generator
\begin{equation}
 \label{eqmainlimgenindbotnet}
\begin{aligned}
& L^{ind}g(DI)=\la u^{ind}(DI)(g(UI)-g(DI))+q^D_{rec}(g(DS)-g(DI)), \\
& L^{ind}g(DS)=\la u^{ind}(DS)(g(US)-g(DS))+q^D_{inf}v_H(g(DI)-g(DS)) \\
& \quad \quad \quad +x_{DI} \be_{DD} (g(DI)-g(DS))+x_{UI} \be_{UD}(g(DI)-g(DS)), \\
& L^{ind}g(UI)=\la u^{ind}(UI)(g(DI)-g(UI))+q^U_{rec}(g(US)-g(UI)), \\
& L^{ind}g(US)=\la u^{ind}(US)(g(DS)-g(US))+q^U_{inf}v_H(g(UI)-g(US)) \\
& \quad \quad \quad +x_{DI} \be_{DU} (g(UI)-g(US))+x_{UI} \be_{UU}(g(UI)-g(US))
\end{aligned}
\end{equation}
depending on the individual control $u^{ind}$.

Assuming that an individual pays a fee $k_D$ per unit of time for the defense system
and $k_I$ per unit time for losses resulting from being infected, her cost during a period of time $T$,
that she tries to minimize, is
\begin{equation}
 \label{eqcostcomputbotnet}
\int_0^T   (k_D \1_D+k_I \1_I) \, ds,
\end{equation}
where $\1_D$ (resp. $\1_I$) is the indicator function of the states $DI,DS$ (resp. of the states $DI$, $UI$).
Assuming that the Herder has to pay $k_H v_H$ per unit of time using efforts $v_H$ and receive the income
$f(x)$ depending on the distribution $x$ of the states of the computers, her payoff, that she tries to maximize, is
\begin{equation}
 \label{eqpayherder}
\int_0^T   (f_H(x)-k_Hv_H) \, ds.
\end{equation}

Therefore, starting with some control
\[
u^{com}=(u^{com}_t(DI), u^{com}_t(DS), u^{com}_t(UI), u^{com}_t(US))
\]
the Herder can find his optimal strategy $v_H(t)$ solving the deterministic optimal control problem with
dynamics \eqref{eqmainkineticbotnet} and payoff \eqref{eqpayherder} finding both optimal $v_H$
and the trajectory $x(t)$. Once $x(t)$ and $v_H(t)$ are known, each individual should solve
the Markov control problem \eqref{eqmainlimgenindbotnet} with costs \eqref{eqcostcomputbotnet}
thus finding the individual optimal strategy
\[
u^{ind}_t=(u^{ind}_t(DI), u^{ind}_t(DS), u^{ind}_t(UI), u^{ind}_t(US)).
\]

 The basic {\it MFG consistency equation} can now be explicitly written as
 \[
 u^{ind}_t=u^{com}_t.
 \]

Instead of analyzing this rather complicated dynamic problem, we shall look for a simpler
problem of consistent stationary strategies.

There are two standard stationary problems naturally linked with a dynamic one,
 one being the search for the average payoff
\[
g =\lim_{T\to \infty} \frac{1}{T}\int_0^T (k_D \1_D+k_I \1_I) \, dt
\]
for long period game, and another the search for discounted optimal payoff.
The first is governed by the solutions of HJB of the form $(T-t)\mu +g$, linear in $t$ (then $\mu$ describing the optimal average payoff),
so that $g$ satisfies the stationary HJB equation:
\begin{equation}
 \label{eqindstHJBbot}
\left\{\begin{aligned}
& \la \min_u u(g(UI)-g(DI))+q^D_{rec}(g(DS)-g(DI))+k_I+k_D=\mu, \\
& \la \min_u u(g(US)-g(DS))+q^D_{inf}v_H(g(DI)-g(DS)) \\
& \quad \quad \quad +x_{DI} \be_{DD} (g(DI)-g(DS))+x_{UI} \be_{UD}(g(DI)-g(DS))+k_D=\mu, \\
& \la \min_u u(g(DI)-g(UI))+q^U_{rec}(g(US)-g(UI))+k_I=\mu, \\
& \la \min_u u(g(DS)-g(US))+q^U_{inf}v_H(g(UI)-g(US)) \\
& \quad \quad \quad +x_{DI} \be_{DU} (g(UI)-g(US))+x_{UI} \be_{UU}(g(UI)-g(US))=\mu
\end{aligned}
\right.
 \end{equation}
 where $\min$ is over two values $\{0,1\}$. We shall denote $u=(u_{DI},u_{UI}, u_{DS}, u_{US})$
the argmax in this solution.

The discounted optimal payoff (with the discounting coefficient $\de$) satisfies the stationary HJB

\begin{equation}
 \label{eqindstHJBbotdisc}
\left\{\begin{aligned}
& \la \min_u u(g(UI)-g(DI)+q^D_{rec}(g(DS)-g(DI))+k_I+k_D=\de g(DI), \\
& \la \min_u u(g(US)-g(DS))+q^D_{inf}v_H(g(DI)-g(DS)) \\
& \quad \quad \quad +x_{DI} \be_{DD} (g(DI)-g(DS))+x_{UI} \be_{UD}(g(DI)-g(DS))+k_D=\de g(DS), \\
& \la \min_u u(g(DI)-g(UI))+q^U_{rec}(g(US)-g(UI))+k_I=\de g(UI), \\
& \la \min_u u(g(DS)-g(US))+q^U_{inf}v_H(g(UI)-g(US)) \\
& \quad \quad \quad +x_{DI} \be_{DU} (g(UI)-g(US))+x_{UI} \be_{UU}(g(UI)-g(US))=\de g(US)
\end{aligned}
\right.
 \end{equation}

The analysis of these two settings is mostly analogous. We shall concentrate on the first one.
Introducing the coefficients
\begin{equation}
 \label{eqdefalphabeta}
 \begin{aligned}
& \al = q^D_{inf}v_H +x_{DI} \be_{DD} +x_{UI} \be_{UD}, \\
& \be =q^U_{inf}v_H +x_{DI} \be_{DU}+x_{UI} \be_{UU},
\end{aligned}
\end{equation}
the stationary HJB equation \eqref{eqindstHJBbot} rewrites as

\begin{equation}
 \label{eqindstHJBbot1}
\left\{
\begin{aligned}
& \la \min(g(UI)-g(DI),0)+q^D_{rec}(g(DS)-g(DI))+k_I+k_D=\mu, \\
& \la \min(g(US)-g(DS),0)+\al (g(DI)-g(DS)) +k_D=\mu, \\
& \la \min(g(DI)-g(UI),0)+q^U_{rec}(g(US)-g(UI))+k_I=\mu, \\
& \la \min(g(DS)-g(US),0)+\be (g(UI)-g(US)) =\mu,
\end{aligned}
\right.
 \end{equation}
 where the choice of the first term as the infimum in these equations corresponds to the choice of control $u=1$.

The {\it stationary MFG consistency} problem is in finding $x=(x_{DI},x_{DS}, x_{UI}, x_{US})$
and $u=(u_{DI},u_{DS}, u_{UI}, u_{US})$, where $x$ is the stationary point of
evolution \eqref{eqmainkineticbotnet}, that is
\begin{equation}
 \label{eqindstfixedbot}
\left\{
\begin{aligned}
& x_{DS} \al -x_{DI} q_{rec}^D+\la (x_{UI}u_{UI}-x_{DI}u_{DI})=0 \\
& -x_{DS} \al +x_{DI} q_{rec}^D+\la (x_{US}u_{US}-x_{DS}u_{DS})=0 \\
& x_{US} \be -x_{UI} q_{rec}^U-\la (x_{UI}u_{UI}-x_{DI}u_{DI})=0 \\
& -x_{US} \be +x_{UI} q_{rec}^U-\la (x_{US}u_{US}-x_{DS}u_{DS})=0,
 \end{aligned}
 \right.
 \end{equation}
with $u=(u_{DI},u_{DS}, u_{UI}, u_{US})$ giving minimum
in the solution to \eqref{eqindstHJBbot} or \eqref{eqindstHJBbot1}. Thus $x$ is a fixed point of the
limiting dynamics of the distribution of large number of agents such that the corresponding
stationary control is individually optimal subject to this distribution. Yet in other words,
$x=(x_{DI},x_{DS}, x_{UI}, x_{US})$
and $u=(u_{DI},u_{DS}, u_{UI}, u_{US})$ solve \eqref{eqindstHJBbot}, \eqref{eqindstfixedbot} simultaneously.

Fixed points can practically model a stationary behavior only if they are stable. Thus we are interested
in {\it stable solutions} $(x,u)$ to the stationary MFG consistency problem \eqref{eqindstfixedbot},\eqref{eqindstHJBbot}, where
a solution is stable if the corresponding stationary distribution $x$ is a stable equilibrium to
\eqref{eqmainkineticbotnet} (with $u$ fixed by this solution).

Apart from stability, the fixed points can be classified via its efficiency.
Namely, let us say that a solution to the stationary MFG is {\it efficient} (or globally optimal)
if the corresponding average cost $\mu$ is minimal among all other solutions.

Talking about strategies, let us reduce the discussion to non-degenerate situations, where the minima in \eqref{eqindstHJBbot1}
are achieved on a single value of $u$ only.
In principle,
there are 16 possible pure stationary strategies (functions from the state space to $\{0,1\}$).
But not all of them can be realized as solutions to \eqref{eqindstHJBbot1}.
In fact if $u_{DI}=1$, then $g(UI)< g(DI)$ (can be equal in degenerate case) and thus $u_{UI}=0$.
This argument forbids all but four strategies as possible solutions to \eqref{eqindstHJBbot1},
namely
\begin{equation}
 \label{eq4strategiesallow}
\left\{
\begin{aligned}
& (i) \quad g(UI)\le g(DI), \quad g(US)\le g(DS) \Longleftrightarrow u_{UI}=u_{US}=0, \quad u_{DI}=u_{DS}=1, \\
& (ii) \quad g(UI)\ge g(DI), \quad g(US)\ge g(DS)  \Longleftrightarrow u_{DI}=u_{DS}=0, \quad u_{UI}=u_{US}=1, \\
& (iii) \quad  g(UI)\le g(DI), \quad g(US)\ge g(DS)  \Longleftrightarrow u_{UI}=u_{DS}=0, \quad  u_{DI}=u_{US}=1,\\
& (iv) \quad g(UI)\ge g(DI), \quad g(US)\le g(DS)  \Longleftrightarrow u_{DI}=u_{US}=0, \quad  u_{UI}=u_{DS}=1.
\end{aligned}
\right.
 \end{equation}

The first two strategies, either always choose $U$ or always choose $D$,
are acyclic, that is the corresponding Markov processes are acyclic in the sense that there does not exist
a cycle in a motion subject to these strategies. Other two strategies choose between $U$ and $D$ differently if infected or not.

Of course, allowing degenerate strategies, more possibilities arise.


To complete the model, let us observe that the natural assumptions on the parameters of the model
arising directly from their interpretation are as follows:

\begin{equation}
\label{eqassumonbotnetsim}
\left\{
\begin{aligned}
& q^D_{rec} \ge q^U_{rec}, \quad q^D_{inf} < q^U_{inf}, \\
& \be_{UD} \le \be_{UU}, \quad  \be_{DD} \le \be_{DU} , \\
& k_D \le k_I.
\end{aligned}
\right.
 \end{equation}

We shall always assume \eqref{eqassumonbotnetsim} hold.
Two additional natural simplifying assumptions that we shall use sometimes are the following:
the infection rate does not depend on the level of defense of
the computer transferring the infection, but only on the level of defence of the susceptible computer, that is,
instead of four coefficients $\be$ one has only 2 of them
\begin{equation}
\label{eqassumonbotnetsim1}
\be_U=\be_{DU}=\be_{UU}, \quad \be_D=\be_{UD}=\be_{DD},
\end{equation}
and the recovery rate do not depend on whether a computer is protected against the infection or not:
\begin{equation}
\label{eqassumonbotnetsim2}
q_{rec}=q^D_{rec}=q^U_{rec}.
\end{equation}

As we shall see, a convenient assumption, which is weaker than \eqref{eqassumonbotnetsim2}, turns out to be
\begin{equation}
\label{eqassumonbotnetsim3}
q^D_{rec}-q^U_{rec} < (q^U_{inf} -q^D_{inf})v_H.
\end{equation}

Finally, it is reasonable to assume that customers can switch rather quickly their regime of defence
(once they are willing to) meaning that we are effectively interested in the asymptotic regime of large $\la$.
As we shall show, in this regime the stationary MFG problem above can be completely solved analytically.
In this sense the present model is more complicated than a related mean-field game model of corruption
with three basic states developed in \cite{KolMalCorr},
where a transparent analytic classification of stable solutions is available already for arbitrary finite $\la$.

\section{Analysis of the stationary HJB equation}
\label{secHJB}

Let us start by solving HJB equation \eqref{eqindstHJBbot1}.

Consider strategy (i) of \eqref{eq4strategiesallow},
so that being unprotected is always optimal. Then \eqref{eqindstHJBbot1} becomes
\begin{equation}
 \label{eqindstHJBbotUnOp1}
\left\{
\begin{aligned}
& \la (g(UI)-g(DI))+q^D_{rec}(g(DS)-g(DI))+k_I+k_D=\mu, \\
& \la (g(US)-g(DS))+\al (g(DI)-g(DS)) +k_D=\mu, \\
& q^U_{rec}(g(US)-g(UI))+k_I=\mu, \\
& \be (g(UI)-g(US)) =\mu.
\end{aligned}
\right.
 \end{equation}
 As the solution $g$ is defined up to an additive constant we can set $g(US)=0$. Then
\eqref{eqindstHJBbotUnOp1} becomes
\begin{equation}
 \label{eqindstHJBbotUnOp2}
\left\{
\begin{aligned}
& \la (g(UI)-g(DI))+q^D_{rec}(g(DS)-g(DI))+k_I+k_D=\mu, \\
& -\la g(DS)+\al (g(DI)-g(DS)) +k_D=\mu, \\
& -q^U_{rec}g(UI)+k_I=\mu, \\
& \be g(UI) =\mu.
\end{aligned}
\right.
 \end{equation}

 From the third and fourth equations we find
 \begin{equation}
 \label{eqindstHJBbotUnOp2a}
 g(UI)=\frac{k_I}{\be +q^U_{rec}}, \quad \mu =\be g(UI)=\frac{\be k_I}{\be +q^U_{rec}}.
 \end{equation}
 Substituting these values in the first and second equations we obtain
 \begin{equation}
 \label{eqindstHJBbotUnOp3}
\left\{
\begin{aligned}
& g(DS)=\frac{k_D-\mu}{\la} +k_I \frac{\al (\be +\la +q^U_{rec})}{\la (\be +q^U_{rec})(\al +\la +q^D_{rec})}, \\
& g(DI)=\frac{k_D-\mu}{\la} +k_I \frac{(\al +\la) (\be +\la +q^U_{rec})}{\la (\be +q^U_{rec})(\al +\la +q^D_{rec})},
\end{aligned}
\right.
 \end{equation}
 and the conditions $g(UI)\le g(DI), g(US)=0\le g(DS)$ become
  \begin{equation}
 \label{eqindstHJBbotUnOp4}
\begin{aligned}
& k_D(\be +q_{rec}^U)(\al +\la + q^D_{rec}) \ge k_I [(\be+\la) q_{rec}^D -(\al+\la)q_{rec}^U], \\
& k_D(\be +q_{rec}^U)(\al +\la + q^D_{rec}) \ge k_I [\be (\la + q_{rec}^D) -\al (\la +q_{rec}^U)]
\end{aligned}
 \end{equation}
 respectively.

Consider strategy (ii) of \eqref{eq4strategiesallow},
so that being defended is optimal. Then \eqref{eqindstHJBbot1} becomes
\begin{equation}
 \label{eqindstHJBbotDeOp1}
\left\{
\begin{aligned}
& q^D_{rec}(g(DS)-g(DI))+k_I+k_D=\mu, \\
& \al (g(DI)-g(DS)) +k_D=\mu, \\
& \la (g(DI)-g(UI))+q^U_{rec}(g(US)-g(UI))+k_I=\mu, \\
& \la (g(DS)-g(US)) + \be (g(UI)-g(US)) =\mu.
\end{aligned}
\right.
 \end{equation}
 Setting $g(DS)=0$ yields
 \begin{equation}
 \label{eqindstHJBbotDeOp2}
\left\{
\begin{aligned}
& - q^D_{rec}g(DI))+k_I+k_D=\mu, \\
& \al g(DI) +k_D=\mu, \\
& \la (g(DI)-g(UI))+q^U_{rec}(g(US)-g(UI))+k_I=\mu, \\
& - \la g(US)) + \be (g(UI)-g(US)) =\mu.
\end{aligned}
\right.
 \end{equation}

 From the first and second equations we find
  \begin{equation}
 \label{eqindstHJBbotDeOp2a}
 g(DI)=\frac{k_I}{\al +q^D_{rec}}, \quad \mu =k_D+\al g(DI)=\frac{\al (k_D+k_I)+k_D q^D_{rec}}{\al +q^D_{rec}}.
 \end{equation}
 Substituting these values in the third and fourth equations we obtain
 \begin{equation}
 \label{eqindstHJBbotDeOp3}
\left\{
\begin{aligned}
& g(US)=-\frac{k_D}{\la} +k_I \frac{\be (\la +q^D_{rec})-\al (\la +q^U_{rec})}{\la (\al +q^D_{rec})(\be +\la +q^U_{rec})}, \\
& g(UI)=-\frac{k_D}{\la} +k_I \frac{(\be +\la)(\la +q^D_{rec})-\al q^U_{rec}}{\la (\al +q^D_{rec})(\be +\la +q^U_{rec})}.
\end{aligned}
\right.
\end{equation}
and the conditions $g(UI)\ge g(DI), \quad g(US)\ge g(DS)=0$ turn to
\begin{equation}
\label{eqindstHJBbotDeOp4}
\begin{aligned}
& k_D(\al +q_{rec}^D)(\be +\la + q^U_{rec}) \le k_I [(\be +\la ) q_{rec}^D -(\al+\la)q_{rec}^U], \\
& k_D(\al +q_{rec}^D)(\be +\la + q^U_{rec}) \le k_I [\be (\la + q_{rec}^D) -\al (\la +q_{rec}^U)]
\end{aligned}
\end{equation}
respectively.

Consider strategy (iii) of \eqref{eq4strategiesallow}.
Then \eqref{eqindstHJBbot1} becomes
\begin{equation}
 \label{eqindstHJBbotMix1Op1}
\left\{
\begin{aligned}
& \la (g(UI)-g(DI))+ q^D_{rec}(g(DS)-g(DI))+k_I+k_D=\mu, \\
& \al (g(DI)-g(DS)) +k_D=\mu, \\
& q^U_{rec}(g(US)-g(UI))+k_I=\mu, \\
& \la (g(DS)-g(US)) + \be (g(UI)-g(US)) =\mu.
\end{aligned}
\right.
 \end{equation}
 Setting $g(DS)=0$ yields
 \[
 \mu=\al g(DI) +k_D
 \]
 from the second equation, then
 \[
 \la g(UI) =g(DI)(\al +\la +q^D_{rec})-k_I
 \]
 from the first equation and
 \[
 g(US)=\frac{g(DI)}{\la q_{rec}^U}[\al \la +q^U_{rec} (\al +\la +q^D_{rec})]+\frac{\la k_D-(\la +q^U_{rec})k_I}{\la q_{rec}^U}
 \]
 from the third one. Plugging these expressions in the fourth equation of \eqref{eqindstHJBbotMix1Op1} we find
 (after many cancelations) $g(DI)$
 and then the other values of $g$:
 \begin{equation}
 \label{eqindstHJBbotMix1Op2}
\left\{
\begin{aligned}
& g(DI)=\frac{(\be +\la +q^U_{rec})(k_I-k_D)}{\al (\be +\la + q^U_{rec})+q^U_{rec} (\al +\la +q^D_{rec})}, \\
& g(US)=\frac{1}{\la}
\frac{k_I [\be(\la +q^D_{rec})-\al (\la +q^U_{rec})]-k_D (\be +q^U_{rec})(\al +\la +q^D_{rec})}
{\al (\be +\la + q^U_{rec})+q^U_{rec} (\al +\la +q^D_{rec})}, \\
&  g(UI)=\frac{1}{\la}
\frac{k_I [(\la +q^D_{rec})(\la +\be)-\al q^U_{rec}]-k_D (\be +\la + q^U_{rec})(\al +\la +q^D_{rec})}
{\al (\be +\la + q^U_{rec})+q^U_{rec} (\al +\la +q^D_{rec})}
\end{aligned}
\right.
 \end{equation}
 Hence
 \[
 \mu=\frac{k_I \al (\be + \la +q^U_{rec})+k_D q^U_{rec}(\al +\la + q^D_{rec})}
{\al (\be +\la + q^U_{rec})+q^U_{rec} (\al +\la +q^D_{rec})}.
\]
The conditions $g(UI)\le g(DI), \quad g(US)\ge g(DS)$ rewrite as
\begin{equation}
 \label{eqindstHJBbotMix1Op3}
\left\{
\begin{aligned}
& k_D (\al+q^D_{rec})(\be +\la +q^U_{rec}) \ge k_I [(\be +\la ) q_{rec}^D -(\al+\la)q_{rec}^U], \\
&  k_D (\al+\la + q^D_{rec})(\be +q^U_{rec}) \le k_I [\be (\la + q_{rec}^D) -\al (\la +q_{rec}^U)].
\end{aligned}
\right.
 \end{equation}

Consider strategy (iv) of \eqref{eq4strategiesallow}.
Then \eqref{eqindstHJBbot1} becomes
\begin{equation}
 \label{eqindstHJBbotMix2Op1}
\left\{
\begin{aligned}
& q^D_{rec}(g(DS)-g(DI))+k_I+k_D=\mu, \\
& \la (g(US)-g(DS)) + \al (g(DI)-g(DS)) +k_D=\mu, \\
& \la (g(DI)-g(UI)) + q^U_{rec}(g(US)-g(UI))+k_I=\mu, \\
& \be (g(UI)-g(US)) =\mu.
\end{aligned}
\right.
 \end{equation}
 Setting $g(US)=0$ yields $\mu=\be g(UI)$ from the fourth equation, then
 \[
 \la g(DI)=g(UI)(\be +\la +q^U_{rec})-k_I
 \]
 from the third equation and
 \[
 \la q^D_{rec} g(DS)=g(UI)[\be \la +q^D_{rec}(\be +\la +q^U_{rec})]-k_D\la -k_I(\la +q^D_{rec})
 \]
 from the first one.
  Plugging these expressions in the second equation of \eqref{eqindstHJBbotMix2Op1} we find $g(UI)$ and then the other values of $g$:
  \begin{equation}
 \label{eqindstHJBbotMix2Op2}
\left\{
\begin{aligned}
& g(UI)= \frac{(k_D+k_I) (\al +\la +q^D_{rec})}
 {\be (\al +\la + q^D_{rec})+q^D_{rec} (\be +\la +q^U_{rec})}, \\
& g(DS)=\frac{1}{\la}
\frac{k_D (\be +\la + q^U_{rec})(\al +q^D_{rec})+ k_I [\al(\la +q^U_{rec})-\be (\la +q^D_{rec})]}
{\be (\al +\la + q^D_{rec})+q^D_{rec} (\be +\la +q^U_{rec})}, \\
&  g(DI)=\frac{1}{\la}
\frac{k_D (\be +\la + q^U_{rec})(\al +\la + q^D_{rec})+k_I [(\al +\la)(\la +q^U_{rec})-\be q^D_{rec}]}
{\be (\al +\la + q^D_{rec})+q^D_{rec} (\be +\la +q^U_{rec})}.
\end{aligned}
\right.
 \end{equation}
 Hence the conditions $g(UI)\ge g(DI), \quad g(DS)\ge g(US)=0$ rewrite as
\begin{equation}
 \label{eqindstHJBbotMix2Op3}
\left\{
\begin{aligned}
& k_D (\al+\la +q^D_{rec})(\be +q^U_{rec}) \le k_I [(\be +\la ) q_{rec}^D -(\al+\la)q_{rec}^U], \\
&  k_D (\al + q^D_{rec})(\be +\la + q^U_{rec}) \ge k_I [\be (\la + q_{rec}^D) -\al (\la +q_{rec}^U)].
\end{aligned}
\right.
 \end{equation}

We are now interested in finding out how many solutions equation \eqref{eqindstHJBbot1} may have for a given $x$.
The first observation in this direction is that the interior of the domain defined by
\eqref{eqindstHJBbotUnOp4} (that is, with a solution of case (i)) and the interior of the domain defined by
\eqref{eqindstHJBbotMix1Op3} (that is, with a solution of case (iii)) do not intersect, because the first inequality
in \eqref{eqindstHJBbotUnOp4} contradicts the second inequality in \eqref{eqindstHJBbotMix1Op3} (apart from the boundary).
Similarly,  the interior of the domain defined by
\eqref{eqindstHJBbotUnOp4} (that is with a solution of case (i)) and the interior of the domain defined by
\eqref{eqindstHJBbotMix2Op3} (that is, with a solution of case (iv)) do not intersect,
and  the interior of the domain defined by
\eqref{eqindstHJBbotDeOp4} (that is, with a solution of case (ii)) does not intersect with the domains having solutions in
cases (iii) or (iv).

Next we find that one can distinguish two natural domains of $x$ classifying the solutions to  HJB equation \eqref{eqindstHJBbot1}:
\[
D_1=\{x : \be+q^U_{rec} > \al+q^D_{rec} \}, \quad D_2=\{x : \be+q^U_{rec} < \al+q^D_{rec} \}.
\]
More explicitly,
\[
D_1=\{x: x_{DI} (\be_{DU}-\be_{DD}) + x_{UI} (\be_{UU}-\be_{UD}) > (q^D_{inf}-q^U_{inf}) v_H +q^D_{rec}-q^U_{rec} \}.
\]

By \eqref{eqassumonbotnetsim} it is seen that under a natural additional simplifying assumptions
\eqref{eqassumonbotnetsim2} or even \eqref{eqassumonbotnetsim3},
all positive $x$ belong to $D_1$ (or its boundary), so that $D_2$ is empty.

Under additional assumption \eqref{eqassumonbotnetsim1} the condition $x\in D_1$ gets a simpler form
\begin{equation}
\label{eqcondD1sim}
x > \bar x =\frac{(q^D_{inf}-q^U_{inf}) v_H +q^D_{rec}-q^U_{rec}}{\be_U-\be_D}.
\end{equation}

To link with the conditions for cases (i)-(iv) one observes the following equivalent forms of the main condition
of being in $D_1$:

\[
\be+q^U_{rec} > \al+q^D_{rec} \Longleftrightarrow
(\be+q^U_{rec})(\al+q^D_{rec}+\la) > (\al+q^D_{rec})(\be+q^U_{rec}+\la)
\]
\begin{equation}
\label{eqcondD1}
\Longleftrightarrow \be (\la + q_{rec}^D) -\al (\la +q_{rec}^U)  > (\be +\la ) q_{rec}^D -(\al+\la)q_{rec}^U.
\end{equation}

From here it is seen that if $x$ belongs simultaneously to the interiors of the domains specified by
\eqref{eqindstHJBbotUnOp4} and \eqref{eqindstHJBbotDeOp4} (that is, with solutions in cases (i) and (ii) simultaneously),
then necessarily $x\in D_1$ (that is, for $x\in D_2$ the conditions specifying cases (i) and (ii) are incompatible).
On the other hand, if  $x$ belongs simultaneously to the interiors of the domains specified by
\eqref{eqindstHJBbotMix1Op3} and \eqref{eqindstHJBbotMix2Op3} (that is, with solutions in cases (iii) and (iv) simultaneously),
then necessarily $x\in D_2$ (that is, for $x\in D_1$ the conditions specifying cases (iii) and (iv) are incompatible).

Denoting $\ka =k_D/k_i$, we can summarize the properties of HJB equation \eqref{eqindstHJBbot1} as follows
(uniqueness is always understood up to the shifts in $g$).

\begin{prop}
\label{propeqHJBbotsim}
Suppose $x\in D_1$.

(1) If
\begin{equation}
\label{eq1propeqHJBbotsim}
\frac{(\be +\la ) q_{rec}^D -(\al+\la)q_{rec}^U}{(\be+q^U_{rec}+\la)(\al+q^D_{rec})}  < \ka
< \frac{\be (\la + q_{rec}^D) -\al (\la +q_{rec}^U)}{(\be+q^U_{rec})(\al+q^D_{rec}+\la)},
\end{equation}
then there exists a unique solution to \eqref{eqindstHJBbot1} belonging to case (iii) and there are no other solutions to \eqref{eqindstHJBbot1}.

(2) If
\begin{equation}
\label{eq2propeqHJBbotsim}
\frac{\be (\la + q_{rec}^D) -\al (\la +q_{rec}^U)}{(\be+q^U_{rec}+\la)(\al+q^D_{rec})}  < \ka
< \frac{(\be +\la ) q_{rec}^D -(\al+\la)q_{rec}^U}{(\be+q^U_{rec})(\al+q^D_{rec}+\la)},
\end{equation}
then there exists a unique solution to \eqref{eqindstHJBbot1} belonging to case (iv) and there are no other solutions to \eqref{eqindstHJBbot1}.

(3) A solution belonging to case (i) exists if and only if
\begin{equation}
\label{eq3propeqHJBbotsim}
 \ka \ge  \frac{\be (\la + q_{rec}^D) -\al (\la +q_{rec}^U)}{(\be+q^U_{rec})(\al+q^D_{rec}+\la)},
\end{equation}
and is unique if this holds.
A solution belonging to case (ii) exists if and only if
\begin{equation}
\label{eq4propeqHJBbotsim}
 \ka \le \frac{(\be +\la ) q_{rec}^D -(\al+\la)q_{rec}^U}{(\be+q^U_{rec}+\la)(\al+q^D_{rec})},
\end{equation}
and is unique if this holds. Either of conditions \eqref{eq3propeqHJBbotsim} or \eqref{eq4propeqHJBbotsim} is incompatible
with either \eqref{eq1propeqHJBbotsim} or \eqref{eq2propeqHJBbotsim}.
In particular, equation \eqref{eqindstHJBbot1} may have at most two solutions (if both \eqref{eq3propeqHJBbotsim} and \eqref{eq4propeqHJBbotsim} hold).

(4) Under \eqref{eqassumonbotnetsim2}, one has always
\begin{equation}
\label{eq6propeqHJBbotsim}
\frac{\be (\la + q_{rec}^D) -\al (\la +q_{rec}^U)}{(\be+q^U_{rec}+\la)(\al+q^D_{rec})}
\ge  \frac{(\be +\la ) q_{rec}^D -(\al+\la)q_{rec}^U}{(\be+q^U_{rec})(\al+q^D_{rec}+\la)},
\end{equation}
and
\begin{equation}
\label{eq5propeqHJBbotsim}
\frac{(\be +\la ) q_{rec}^D -(\al+\la)q_{rec}^U}{(\be+q^U_{rec}+\la)(\al+q^D_{rec})}
\le \frac{\be (\la + q_{rec}^D) -\al (\la +q_{rec}^U)}{(\be+q^U_{rec})(\al+q^D_{rec}+\la)}.
\end{equation}
Hence \eqref{eq2propeqHJBbotsim} becomes impossible and conditions \eqref{eq3propeqHJBbotsim}  and \eqref{eq4propeqHJBbotsim}
become incompatible implying the uniqueness of the solution to \eqref{eqindstHJBbot1} for any $x\in D_1$.
This unique solution belongs to
cases (ii), (iii) and (i) respectively for $\ka$ satisfying \eqref{eq4propeqHJBbotsim}, \eqref{eq1propeqHJBbotsim},
\eqref{eq3propeqHJBbotsim} (when equality holds in \eqref{eq3propeqHJBbotsim} or \eqref{eq4propeqHJBbotsim}, two solutions
from different cases become coinciding).
\end{prop}

\begin{proof}
Statements (1)-(3) follow from the arguments given above. (iv) Under \eqref{eqassumonbotnetsim2}, conditions \eqref{eq5propeqHJBbotsim} and \eqref{eq6propeqHJBbotsim} rewrite as
\[
q_{rec}(\be-\al)^2-(\be-\al) (\be+\la +q_{rec})(\al +q_{rec}) \le 0
\]
and
 \[
 q_{rec} (\be-\al)^2 +(\be-\al) (\be+q_{rec})(\al +\la +q_{rec}) \ge 0,
\]
which obviously hold.
\end{proof}

\begin{remark}
\label{remarkonvarioussol}
(1) When \eqref{eqassumonbotnetsim2} does not hold one can find situations when solutions from cases (i) and (ii) exist simultaneously.
To get simple examples one can assume $\ka =1$. (2) When two solutions exist simultaneously one can discriminate them by the values
of the average payoff $\mu$. One sees from  \eqref{eqindstHJBbotUnOp2a} and \eqref{eqindstHJBbotDeOp2a}, that
$\mu$ arising from cases (i) and (ii) are different (apart from a single value of $\ka$).
(3) The uniqueness result under \eqref{eqassumonbotnetsim2} is quite remarkable, as it does not seem to follow a priori
from any intuitive arguments.
\end{remark}

Again directly from the argument above one can conclude the following.
\begin{prop}
\label{propeqHJBbotsimD2}
Suppose $x\in D_2$.

(1) If
\begin{equation}
\label{eq1propeqHJBbotsimD2}
\ka > \frac{(\be +\la ) q_{rec}^D -(\al+\la)q_{rec}^U}{(\be+q^U_{rec})(\al+q^D_{rec}+\la)},
\end{equation}
then there exists a unique solution to \eqref{eqindstHJBbot1} belonging to case (i) and there are no other solutions to \eqref{eqindstHJBbot1}.

(2) If
\begin{equation}
\label{eq2propeqHJBbotsimD2}
\ka < \frac{\be (\la + q_{rec}^D) -\al (\la +q_{rec}^U)}{(\be+q^U_{rec}+\la)(\al+q^D_{rec})},
\end{equation}
then there exists a unique solution to \eqref{eqindstHJBbot1} belonging to case (ii) and there are no other solutions to \eqref{eqindstHJBbot1}.

(3) A solution belonging to case (iii) exists if and only if
\begin{equation}
\label{eq3propeqHJBbotsimD2}
\frac{(\be +\la ) q_{rec}^D -(\al+\la)q_{rec}^U}{(\be+q^U_{rec}+\la)(\al+q^D_{rec})}
\le \ka \le  \frac{\be (\la + q_{rec}^D) -\al (\la +q_{rec}^U)}{(\be+q^U_{rec})(\al+q^D_{rec}+\la)},
\end{equation}
and is unique if this holds.
A solution belonging to case (iv) exists if and only if
\begin{equation}
\label{eq4propeqHJBbotsimD2}
\frac{\be (\la + q_{rec}^D) -\al (\la +q_{rec}^U)}{(\be+q^U_{rec}+\la)(\al+q^D_{rec})}
 \le \ka \le \frac{(\be +\la ) q_{rec}^D -(\al+\la)q_{rec}^U}{(\be+q^U_{rec})(\al+q^D_{rec}+\la)},
\end{equation}
and is unique if this holds. Either of conditions \eqref{eq3propeqHJBbotsimD2} or \eqref{eq4propeqHJBbotsimD2} is incompatible
with either \eqref{eq1propeqHJBbotsimD2} or \eqref{eq2propeqHJBbotsimD2}.
In particular, equation \eqref{eqindstHJBbot1} may have at most two solutions
(if \eqref{eq3propeqHJBbotsimD2} and \eqref{eq4propeqHJBbotsimD2} hold simultaneously).
\end{prop}

Essential simplifications that allow eventually for a full classification of the stationary MFG consistency problem
occur in the limit of large $\la$.
For a precise formulation in case
\begin{equation}
\label{eqassumonbotnetsim4}
\de= q^D_{rec}-q^U_{rec} >0
\end{equation}
one needs further decomposition of the domains $D_1,D_2$. Namely,
for $j=1,2$, let
\[
D_{j1}=\{x\in D_j : \frac{\de}{\al +q^D_{rec}} < \frac{\be-\al} {\be +q^U_{rec}} \}.
\]

\begin{prop}
\label{propHJBbotsimlargela}
The following hold for large $\la$ outside
an interval of $\ka$ of size of order $\la^{-1}$:

(1) Under \eqref{eqassumonbotnetsim2}
conditions \eqref{eq4propeqHJBbotsim}, \eqref{eq1propeqHJBbotsim},
\eqref{eq3propeqHJBbotsim} classifying the solutions to the HJB equation rewrite as
\begin{equation}
\label{eq1propHJBbotsimlargela}
\ka \le 0, \quad 0 < \ka < \frac{(\be-\al)}{\be +q},
\quad \ka \ge \frac{(\be-\al)}{\be +q},
\end{equation}
respectively. In particular, solutions of case (ii) become impossible.

(2) Suppose $x\in D_1$ and \eqref{eqassumonbotnetsim4} holds.
If $x\in D_{11}$, there exists a unique solution to \eqref{eqindstHJBbot1}, which belongs to cases (ii), (iii), (i) for
\begin{equation}
\label{eq2propHJBbotsimlargela}
\ka < \frac{\de}{\al +q^D_{rec}}, \quad \frac{\de}{\al +q^D_{rec}} <\ka < \frac{\be-\al} {\be +q^U_{rec}},
\quad \ka > \frac{\be-\al} {\be +q^U_{rec}},
\end{equation}
respectively.
If $x\in D_{12}$, solutions from case (iii) do not exist and there exist two solutions
to \eqref{eqindstHJBbot1} for
\begin{equation}
\label{eq3propHJBbotsimlargela}
\frac{\be-\al} {\be +q^U_{rec}} < \ka < \frac{\de}{\al +q^D_{rec}},
\end{equation}
 belonging to cases (i) and (ii),
and only one solution otherwise.

(3) Suppose $x \in D_2$.
If $x\in D_{22}$, solutions from case (iii) do not exist and there is always a unique solution
to \eqref{eqindstHJBbot1} belonging to case (ii), (iv) or (i), for
\begin{equation}
\label{eq4propHJBbotsimlargela}
\ka < \frac{\be-\al} {\al +q^D_{rec}}, \quad \frac{\be-\al} {\al +q^D_{rec}} < \ka < \frac{\de}{\be +q^U_{rec}},
\quad \ka >  \frac{\de}{\be +q^U_{rec}},
\end{equation}
respectively.
If $x\in D_{21}$, then there are two solutions to \eqref{eqindstHJBbot1} for
\begin{equation}
\label{eq5propHJBbotsimlargela}
\frac{\de}{\al +q^D_{rec}} < \ka < \frac{\be-\al} {\be +q^U_{rec}},
\end{equation}
which belong to cases (iii) and (iv),
and one solution otherwise. This unique solution belongs to case (ii) or (i) for
\[
\ka <  \frac{\be-\al} {\al +q^D_{rec}}, \quad \ka > \frac{\de}{\be +q^U_{rec}}
\]
respectively and to case (iv) otherwise.
\end{prop}

\begin{proof}
Statement (ii) follows from the observation that, for $\de>0$ and large $\la$,
conditions \eqref{eq1propeqHJBbotsim} - \eqref{eq2propeqHJBbotsim} turn to
\begin{equation}
\label{eq1propeqHJBbotsimlargela}
\frac{\de}{\al+q^D_{rec}}  < \ka
< \frac{\be -\al }{\be+q^U_{rec}}, \quad \frac{\be -\al }{\al+q^D_{rec}} <\ka < \frac{\de}{\be+q^U_{rec}}
\end{equation}
respectively, and
conditions \eqref{eq3propeqHJBbotsim} - \eqref{eq4propeqHJBbotsim} turn to
\begin{equation}
\label{eq3propeqHJBbotsimlargela}
 \ka \ge \frac{\be -\al }{\be+q^U_{rec}}, \quad \ka \le \frac{\de}{\al+q^D_{rec}}
\end{equation}
respectively. Other statements are similar.
\end{proof}

\section{Analysis of the fixed points}
\label{secfixedpoint}

Next we are solving the fixed point system \eqref{eqindstfixedbot}.

In case (i), that is with $u_{UI}=u_{US}=0, u_{DI}=u_{DS}=1$, equation \eqref{eqindstfixedbot} takes the form
\begin{equation}
 \label{eqindstfixedbotUnOp1}
\left\{
\begin{aligned}
& x_{DS} \al -x_{DI} q_{rec}^D-\la x_{DI}=0 \\
& -x_{DS} \al +x_{DI} q_{rec}^D-\la x_{DS}=0 \\
& x_{US} \be -x_{UI} q_{rec}^U+\la x_{DI}=0 \\
& -x_{US} \be +x_{UI} q_{rec}^U+\la x_{DS}=0.
\end{aligned}
\right.
\end{equation}
Adding the first two equations we get $x_{DI}=x_{DS}=0$, and the system reduces to the single equation
\[
x_{US} \be -x_{UI} q_{rec}^U=0.
\]
Substituting the value of $\be$ yields
\[
x_{US} (q^U_{inf}v_H +x_{UI}\be_{UU}) -x_{UI} q_{rec}^U=0.
\]
Denoting $y=x_{UI}$ it follows that $x_{US}=1-y$ and thus
\[
Q_U(y)= \be_{UU}y^2+y(q^U_{rec} -\be_{UU} +q^U_{inf} v_H) -q^U_{inf} v_H=0.
\]
This equation has a unique solution on the interval $(0,1)$:
\begin{equation}
\label{eqindstfixedbotUnOp2}
x^*=x^*_{UI}=\frac{1}{2\be_{UU}}\left[ \be_{UU}-q^U_{rec}-q^U_{inf} v_H
+\sqrt{(\be_{UU} +q^U_{inf} v_H)^2+(q^U_{rec})^2-2 q^U_{rec} (\be_{UU} -q^U_{inf} v_H)}\right].
\end{equation}

The stability of the fixed point $x=(0,0, x^*, 1-x^*)$ means its stability as a fixed point of the dynamics
\begin{equation}
 \label{eqindstfixedbotUnOp3}
\left\{
\begin{aligned}
& \dot x_{DI} = x_{DS} \al -x_{DI} q_{rec}^D-\la x_{DI} \\
& \dot x_{DS} = -x_{DS} \al +x_{DI} q_{rec}^D-\la x_{DS} \\
& \dot x_{UI} = x_{US} \be -x_{UI} q_{rec}^U+\la x_{DI} \\
& \dot x_{US} = -x_{US} \be +x_{UI} q_{rec}^U+\la x_{DS}.
\end{aligned}
\right.
\end{equation}

We rewrite it by shifting the variables by the value of the stationary point, that is, in terms of
$x_{DI}, x_{DS}, y=x_{UI}-x^*, z=x_{US}-(1-x^*)$. Since the sum of these variables is one,
we have effectively the system of three equations on the variables
$x_{DI}, x_{DS}, y$:
\[
\left\{
\begin{aligned}
& \dot x_{DI} = [q^D_{inf} v_H +x_{DI} \be_{DD} +(x^* +y)\be_{UD}]x_{DS} -(\la +q^D_{rec})x_{DI} \\
& \dot x_{DS} = -[q^D_{inf} v_H +x_{DI} \be_{DD} +(x^* +y)\be_{UD}]x_{DS} + q_{rec}^D x_{DI} -\la x_{DS} \\
& \dot y = (1-x^*-y-x_{DI}-x_{DS})[q^U_{inf} v_H +\be_{DU} x_{DI}+\be_{UU} (y+x^*)] -(y+x^*)q^U_{rec}+\la x_{DS}.
\end{aligned}
\right.
\]
Its linear approximation around the fixed point $(0,0,0)$ is
\[
\left\{
\begin{aligned}
& \dot x_{DI} = -(\la +q^D_{rec})x_{DI} + (q^D_{inf} v_H +x^* \be_{UD})x_{DS} \\
& \dot x_{DS} = q_{rec}^D x_{DI} -(q^D_{inf} v_H +x^* \be_{UD}+\la)x_{DS} \\
& \dot y = (1-x^*)[\be_{DU} x_{DI}+\be_{UU} y] -(y+x_{DI}+x_{DS})(q^U_{inf} v_H +x^* \be_{UU})- q^U_{rec} y +\la x_{DS},
\end{aligned}
\right.
\]
and the corresponding characteristic equation for the eigenvalues $\xi$ is
\[
[(1-x^*)\be_{UU} -(q^U_{inf}v_H +x^*\be_{UU})-q^U_{rec}-\xi]
\]
\[
\times
[(\xi+\la +q^U_{inf}v_H +x^*\be_{UU})(\xi +\la +q^D_{rec})-q^D_{rec}(q^U_{inf}v_H +x^*\be_{UU})]=0.
\]
The free term cancels in the second multiplier and we get the eigenvalues
\begin{equation}
 \label{eqindstfixedbotUnOp4}
\left\{
\begin{aligned}
& \xi_1 = (1-x^*)\be_{UU} - q^U_{inf}v_H-x^* \be_{UU}-q^U_{rec} \\
& \xi_2 = -\la -(q_{rec}^D+q^U_{inf}v_H+x^* \be_{UU}) \\
& \xi_3 = -\la.
\end{aligned}
\right.
\end{equation}
The second and the third eigenvalues being negative, the condition of stability is reduced to the negativity of the first eigenvalue,
that is, to the condition
\begin{equation}
 \label{eq2propUDop}
2x^*> 1-\frac{q^U_{rec}+q^U_{inf}v_H}{\be_{UU}}.
\end{equation}
But it always holds for $x^*$ of form \eqref{eqindstfixedbotUnOp2}.

Thus we proved the first part of the following statement and the second is analogous.

\begin{prop}
\label{propFixUDop}
(1) There exists a unique solution to system \eqref{eqindstfixedbot} with the strategy $U$ being individually optimal
(that is, with the first acyclic stationary strategy $u_{UI}=u_{US}=0, u_{DI}=u_{DS}=1$) and it is stable.
It equals $x=(0,0, x^*_{UI}, 1-x^*_{UI})$ with $x^*_{UI}$ given by \eqref{eqindstfixedbotUnOp2}.

(2)   There exists a unique solution to system \eqref{eqindstfixedbot} with the strategy $D$ being individually optimal
(that is, with the second acyclic stationary strategy) and it is stable.
It equals $x=(x^*_{DI}, 1-x^*_{DI},0,0)$ with $x^*_{DI}$ being the unique solution of equation
\begin{equation}
 \label{eq3propUDop}
Q_D(y)= \be_{DD}y^2+y(q^D_{rec} -\be_{DD} +q^D_{inf} v_H) -q^D_{inf} v_H=0
\end{equation}
on the interval $(0,1)$, that is
\[
x^*_{DI}=\frac{1}{2\be_{DD}}\left[ \be_{DD}-q^D_{rec}-q^D_{inf} v_H
+\sqrt{(\be_{DD} +q^D_{inf} v_H)^2+(q^D_{rec})^2-2 q^D_{rec} (\be_{DD} -q^D_{inf} v_H)}\right].
\]
\end{prop}

Let us consider
case (iii):
$u_{UI}=u_{DS}=0, u_{DI}=u_{US}=1$. Then \eqref{eqindstfixedbot} takes the form
\begin{equation}
 \label{eqindstfixedbotMixOp1}
\left\{
\begin{aligned}
& x_{DS} \al -x_{DI} q_{rec}^D-\la x_{DI}=0 \\
& -x_{DS} \al +x_{DI} q_{rec}^D+\la x_{US}=0 \\
& x_{US} \be -x_{UI} q_{rec}^U+\la x_{DI}=0 \\
& -x_{US} \be +x_{UI} q_{rec}^U-\la x_{US}=0.
\end{aligned}
\right.
\end{equation}

By adding the first two equations we get $x_{DI}=x_{US}$
with two independent equations left:

\begin{equation}
 \label{eqindstfixedbotMixOp2}
\left\{
\begin{aligned}
& x_{DS} \al -(q_{rec}^D +\la) x_{DI}=0 \\
& x_{DI} (\be +\la) -x_{UI} q_{rec}^U=0.
\end{aligned}
\right.
\end{equation}
This rewrites as two equations on the two independent variables $x_{DI}, x_{UI}$ as
 \begin{equation}
 \label{eqindstfixedbotMixOp3}
\left\{
\begin{aligned}
& (1- x_{UI} -2x_{DI}) (q^D_{inf} v_H +\be_{DD} x_{DI} +\be_{UD} x_{UI}) -(q_{rec}^D +\la) x_{DI}=0 \\
& x_{DI} (q^U_{inf} v_H +\be_{DU} x_{DI} +\be_{UU} x_{UI} +\la) -x_{UI} q_{rec}^U=0.
\end{aligned}
\right.
\end{equation}
Solving the second equation with respect to $x_{UI}$,
 \begin{equation}
 \label{eqindstfixedbotMixOp4}
x_{UI}=\frac{x_{DI}(q^U_{inf} v_H +x_{DI} \be_{DU}+\la)}{q^U_{rec} -\be_{UU} x_{DI}},
\end{equation}
 and substituting in the first one,
leads to a fourth order equation on $y=x_{DI}$.
This equation does not seem to be much revealing in general.
Of course it can be fully analyzed by numeric methods, but we shall turn now to the large $\la$ asymptotics
that yields more manageable results.

For large $\la$ we get directly from \eqref{eqindstfixedbotMixOp4} that
\[
x_{UI}=\frac{x_{DI}\la}{q^U_{rec} -\be_{UU} x_{DI}}(1+O(\la^{-1})).
\]
But this implies that $x_{DI}$ is small of order $O(\la^{-1})$, so that
\begin{equation}
\label{eqindstfixedbotMixOp5}
x_{UI}=\frac{x_{DI}\la}{q^U_{rec}}(1+O(\la^{-1})), \quad x_{DI}=\frac{x_{UI}q^U_{rec}}{\la}(1+O(\la^{-1})).
\end{equation}
Substituting this in the first equation of \eqref{eqindstfixedbotMixOp3} yields
 \begin{equation}
 \label{eqindstfixedbotMixOp6}
\be_{UD}x_{UI}^2+x_{UI}(q^U_{rec}-\be_{UD}+q^D_{inf}v_H)-q^D_{inf}v_H =O(\la^{-1}),
\end{equation}
which is of the same type as equations \eqref{eq3propUDop} up to terms of order $\la^{-1}$
(and coincides with it under \eqref{eqassumonbotnetsim1}, \eqref{eqassumonbotnetsim2}).
Therefore, for large $\la$, there exists a unique solution to \eqref{eqindstfixedbotMixOp6}
from the interval $(0,1)$:
\[
\bar x^*_{UI}=O(\la^{-1})
\]
 \begin{equation}
 \label{eqindstfixedbotMixOp6a}
+\frac{1}{2\be_{UD}}\left[ \be_{UD}-q^U_{rec}-q^D_{inf} v_H
+\sqrt{(\be_{UD} +q^D_{inf} v_H)^2+(q^U_{rec})^2-2 q^U_{rec} (\be_{UD} -q^D_{inf} v_H)}\right].
\end{equation}

The stability of the fixed point $x=(x^*_{DI},x^*_{DS}=1-\bar x^*_{UI}-2x^*_{DI}, \bar x^*_{UI}, x^*_{US}=x^*_{DI})$
means its stability as a fixed point of the dynamics
\begin{equation}
 \label{eqindstfixedbotMixOp7}
\left\{
\begin{aligned}
& \dot x_{DI} = x_{DS} \al -x_{DI} q_{rec}^D-\la x_{DI} \\
& \dot x_{DS} = -x_{DS} \al +x_{DI} q_{rec}^D+\la x_{US} \\
& \dot x_{UI} = x_{US} \be -x_{UI} q_{rec}^U+\la x_{DI} \\
& \dot x_{US} = -x_{US} \be +x_{UI} q_{rec}^U-\la x_{US}.
\end{aligned}
\right.
\end{equation}

In terms of independent variables
\[
\tilde x_{DI}=x_{DI}-x_{DI}^*, \quad \tilde x_{US}=x_{US}-x^*_{US}, \quad y=\tilde x_{UI}=x_{UI}-\bar x^*_{UI}
\]
this rewrites as
\begin{equation}
 \label{eqindstfixedbotMixOp8}
\left\{
\begin{aligned}
& \frac{d}{dt}\tilde x_{DI} = (1-y-\bar x_{UI}^*-\tilde x_{DI}-\tilde x_{US})\al -\tilde x_{DI} q_{rec}^D-\la (\tilde x_{DI}+x_{DI}^*)+O(\la^{-1}) \\
& \dot y = \tilde x_{US}\be  -(y+\bar x^*_{UI}) q_{rec}^U+\la  (\tilde x_{DI}+x_{DI}^*) +O(\la^{-1})\\
& \frac{d}{dt} \tilde x_{US} = -\tilde x_{US} \be +(y+\bar x^*_{UI}) q_{rec}^U-\la (\tilde x_{US}+x_{US}^*) +O(\la^{-1}).
\end{aligned}
\right.
\end{equation}
with
\[
\al =q^D_{inf} v_H +\tilde x_{DI} \be_{DD} +(y+\bar x^*_{UI})\be_{UD}+O(\la^{-1}),
\]
\[
\be =q^U_{inf} v_H +\tilde x_{DI} \be_{DU} +(y+\bar x^*_{UI})\be_{UU}+O(\la^{-1}).
\]

Linearized around the fixed point $(0,0,0)$ system \eqref{eqindstfixedbotMixOp8} takes the form
\[
\left\{
\begin{aligned}
& \frac{d}{dt} \tilde x_{DI} = -(y+\tilde x_{DI}+\tilde x_{US})(q^D_{inf} v_H +\bar x^*_{UI}\be_{UD})
+(1-\bar x^*_{UI})(\tilde x_{DI}\be_{DD} +\tilde x_{UI} \be_{UD})-\tilde x_{DI}(q^D_{rec}+\la)  \\
& \dot y = \tilde x_{US} (q^U_{inf} v_H +\bar x^*_{UI}\be_{UU})-\tilde x_{UI} q^U_{rec} +\la \tilde x_{DI} \\
& \frac{d}{dt} \tilde x_{US} = -\tilde x_{US} (q^U_{inf} v_H +\bar x^*_{UI}\be_{UU})+\tilde x_{UI} q^U_{rec} -\la \tilde x_{US}
\end{aligned}
\right.
\]
up to terms of order $O(\la^{-1})$.
Thus the matrix of the linear approximation divided by $\la$ is
\[
M(\la)=\left(
\begin{aligned}
& O(\la^{-1})-1 \quad [-q^D_{rec}v_H+\be_{UD}(1-2\bar x^*_{UI})]/\la +O(\la^{-2})  \quad \quad O(\la^{-1}) \\
& O(\la^{-1})+1 \quad \quad \quad -q^U_{rec}/\la +O(\la^{-2}) \quad \quad \quad \quad \quad \quad \quad \quad \quad O(\la^{-1}) \\
&  O(\la^{-1}) \quad \quad \quad \quad \quad \quad \quad O(\la^{-1}) \quad \quad \quad \quad \quad \quad \quad \quad \quad \quad \quad O(\la^{-1})-1
\end{aligned}
\right).
\]
The first order approximation of this matrix in $\la^{-1}$ is
\[
 M_0=\left(
\begin{aligned}
& -1 \quad 0  \quad \quad 0 \\
& \quad 1 \quad 0 \quad \quad 0 \\
& \quad 0 \quad 0 \quad -1
\end{aligned}
\right).
\]
and has eigenvalue $-1$ of double multiplicity and a zero eigenvalue.
Hence all eigenvalues of $M(\la)$ are negative if and only if its determinant $\det (M(\la))$ is negative.
As seen directly
\[
\det (M(\la)) = [-q^U_{rec}-q^D_{rec}v_H+\be_{UD}(1-2\bar x^*_{UI})]/\la + O(\la^{-2}),
\]
and is negative for large $\la$ if and only if
\[
\be_{UD}(2\bar x^*_{UI}-1)>q^U_{rec}+q^D_{rec}v_H,
\]
which always holds by
\eqref{eqindstfixedbotMixOp6a}.
Thus we proved the first part of the following statement and the second part is analogous.

\begin{prop}
\label{propFixMixop}
(1) For large $\la$ there exists a unique solution to system \eqref{eqindstfixedbot} in case (iii), that is with
$u_{UI}=u_{DS}=0, u_{DI}=u_{US}=1$, and it is stable.
It has the form $x=(0, 1-\bar x^*_{UI}, \bar x^*_{UI},0)$ up to corrections of order $\la^{-1}$,
with $\bar x^*_{UI}$ being the unique solution of equation \eqref{eqindstfixedbotMixOp6} on $(0,1)$ given by \eqref{eqindstfixedbotMixOp6a}.

(2) For large $\la$ there exists a unique solution to system \eqref{eqindstfixedbot} in case (iv), that is with
$u_{UI}=u_{DS}=1, u_{DI}=u_{US}=0$, and it is stable.
It has the form $x=(\bar x^*_{DI},0,0,1- \bar x^*_{UI})$ up to corrections of order $\la^{-1}$,
with $\bar x^*_{DI}$ being the unique solution of equation
 \begin{equation}
 \label{eqindstfixedbotMixOp9}
\be_{DU}x_{DI}^2+x_{DI}(q^U_{rec}-\be_{DU}+q^U_{inf}v_H)-q^U_{inf}v_H =O(\la^{-1}),
\end{equation}
 on $(0,1)$.
\end{prop}

\section{Solutions to the stationary MFG problem}
\label{secMFG}

Combining Propositions \ref{propFixUDop}, \ref{propFixMixop} and \ref{propHJBbotsimlargela}  allows one to fully characterize
the solutions to our stationary MFG consistency problem for large $\la$.

The most straightforward general conclusion is the following.
\begin{theorem}
\label{thstatbotnetMFGsim0}
For large $\la$ there may exist up to 4 solutions to the stationary MFG problem,
with only one in each of the cases (i) -(iv). All these solutions are stable.
\end{theorem}

\begin{remark}
Notice that already this statement is not at all obvious a priori,
and may not be true for finite $\la$, where solutions to case (iii) or (iv) are found from an equation of fourth order.
\end{remark}

As an example of more precise classification, let us present it under assumption \eqref{eqassumonbotnetsim3}
that ensures that all solutions lie in the domain $D_1$.

Let us introduce the function
\[
\ka (z)= \frac{(q^U_{inf}-q^D_{inf})v_H+z(\be_{UU}-\be_{UD})}{q^U_{inf}v_H+z\be_{UU}+q}.
\]

First let \eqref{eqassumonbotnetsim2} hold.
It is seen from Propositions \ref{propFixUDop} and \ref{propFixMixop}
that for large $\la$, (and apart from $\ka$ from negligible intervals of size of order $\la^{-1}$
that we shall ignore), a solution of the stationary MFG problem exists in case (i) if
\[
\ka > \ka^* = \ka (x_{UI}^*),
\]
and a solution of the stationary MFG problem exists in case (iii) if
\[
\ka < \bar \ka^* = \ka (\bar x_{UI}^*),
\]
where $x^*_{UI}$ and $\bar x_{UI}^*$ are given by \eqref{eqindstfixedbotUnOp2} and \eqref{eqindstfixedbotMixOp6a} respectively.
Thus one can have up to two (automatically stable) solutions to the stationary MFG problem.
Let us make this number precise.

Differentiating $\ka(z)$ we find directly that it is increasing if
\begin{equation}
 \label{eqcondincreaskappa}
\be_{UU}(q^D_{inf}v_H+q) > \be_{UD}(q^U_{inf}v_H+q),
\end{equation}
and decreasing otherwise. Hence the relation $\ka^* > \bar \ka^*$ is equivalent to the same or the opposite
relation for $x^*_{UI}$ and $\bar x^*_{UI}$. Thus we are led to the following conclusion.

\begin{theorem}
\label{thstatbotnetMFGsim1}
Let \eqref{eqassumonbotnetsim2} hold.

(1) If \eqref{eqcondincreaskappa} holds and $x^*_{UI} >\bar x^*_{UI}$, or
the opposite to \eqref{eqcondincreaskappa} holds and $x^*_{UI} <\bar x^*_{UI}$, then $\ka^* >\bar \ka^*$.
Consequently, for $\ka <\bar \ka^*$ there exists a unique solution to the stationary MFG problem, which is stable
and belongs to case (iii); for $\ka \in (\bar \ka^*, \ka^*)$ there are no solutions to the stationary MFG problem;
for $\ka > \ka^*$ there exists a unique solution to the stationary MFG problem, which is stable
and belongs to case (i).

(2) If \eqref{eqcondincreaskappa} holds and $x^*_{UI} <\bar x^*_{UI}$, or
the opposite to \eqref{eqcondincreaskappa} holds and $x^*_{UI} >\bar x^*_{UI}$, then $\ka^* <\bar \ka^*$.
Consequently, for $\ka <\ka^*$ there exists a unique solution to the stationary MFG problem, which is stable
and belongs to case (iii); for $\ka \in (\ka^*, \bar \ka^*)$ there exist two (stable) solutions to the stationary MFG problem;
for $\ka > \bar \ka^*$ there exists a unique solution to the stationary MFG problem, which is stable
and belongs to case (i).
\end{theorem}

Thus if one considers the system for all parameters fixed except for $\ka$ (essentially specifying the price of the defence service),
points $\ka^*$ and $\bar \ka^*$ are the bifurcation points, where the phase transitions occur.

To deal with case when  \eqref{eqassumonbotnetsim2} does not hold let us introduce the numbers
\[
\ka_1=\frac{\be -\al}{\be+q^U_{rec}}(x^*_{UI}), \quad \ka_2=\frac{\de }{\al +q^D_{rec}}(x^*_{DI}),
\quad \ka_3=\frac{\de}{\al+q^D_{rec}}(\bar x^*_{UI}), \quad \ka_4=\frac{\be -\al}{\be+q^U_{rec}}(\bar x^*_{UI}),
\]
where $x^*_{UI}, x^*_{DI}, \bar x^*_{UI}$ in brackets mean that $\al, \be$ defined in \eqref{eqdefalphabeta} are evaluated
at the corresponding solutions given by Propositions \ref{propFixUDop} and \ref{propFixMixop}.
Since $x^*_{UI}, x^*_{DI}, \bar x^*_{UI}$ are expressed in terms of different parameters,
any order relation between them are to be expected in general. Of course, restrictions appear under additional assumptions,
for instance $x^*_{DI}=\bar x^*_{UI}$ under \eqref{eqassumonbotnetsim1}.
From Proposition \ref{propHJBbotsimlargela} we deduce the following.

\begin{theorem}
\label{thstatbotnetMFGsim2}
Let \eqref{eqassumonbotnetsim3} and \eqref{eqassumonbotnetsim4} hold.

Depending on the order relation between $x^*_{UI}, x^*_{DI}, \bar x^*_{UI}$, one can have up to $3$ solutions
to  the stationary MFG problem for large $\la$, the characterization in each case being fully explicit,
since for $\ka >\ka_1$, there exists a unique  solution in case (i),
for $\ka <\ka_2$, there exists a unique  solution in case (ii),
for $\ka_3 < \ka < \ka_4$, there exists a unique  solution in case (iii).

\end{theorem}

Thus in this case the points $\ka_1, \ka_2, \ka_3, \ka_4$ are the bifurcation points, where the phase transitions occur.

The situation when \eqref{eqassumonbotnetsim3} does not hold is analogous, though there appears an additional bifurcation
relating to $x$ crossing the border between $D_1$ and $D_2$, and the possibility of having four solutions arises.

\section{Discussion}

Our model of four basic states is of course the simplest one that takes effective account of both interaction (infection)
and rational decision making. It suggests extensions in various directions.
For instance, it is practically important to allow for the choice of various competing protection systems, leading
to a model with $2d$ basic states: $iI$ and $iS$, where $i\in \{1, \cdots , d\}$ denotes the $i$th defense system
available (which can be alternatively interpreted as the levels of protection provided by a single or different firms),
while $S$ and $I$ denote again susceptible or infected state, with all other parameters depending on $i$.
On the other hand, in the spirit of papers \cite{Kolpresres}, \cite{Ko12} that concentrate on modeling myopic behavior
(rather than rational optimization) of players one can consider the set of computer owners consisting of two groups,
rational optimizers and those changing their strategies by copying their neighbors.

The main theoretical question arising from our results concerns the rigorous relation between stationary and
dynamic MFG solutions, which in general is in front of research in the mean-field game literature.
We hope that working with our simple model with fully solved stationary version can help to get new insights in this direction.
In the present context the question can be formulated as follows. Suppose that, if at some moment of time
$N$ players are distributed according certain
frequency vector $x$ among the four basic state, each player chooses the optimal strategy $u$
arising from the solution of the stationary problem for fixed $x$ (fully described in Section \ref{secHJB}),
and the Markov evolution continues according to the generator $L$.
When two solutions are available, players may be supposed to choose the one with the lowest $\mu$ (see
see Remark \ref{remarkonvarioussol} (2)).
The resulting changes in $x$
induce the corresponding changes of $u$
specifying a well-defined Markov process on the states of $N$ agents. Intuitively, we would expect this
evolution stay near our stationary MFG solutions for large $N$ and $t$. Can one prove something like that?

\end{document}